\documentclass{amsart}
\usepackage{amsmath,amsthm,amsfonts,amssymb,amscd,mathrsfs}
\usepackage{psfrag,graphicx}

\usepackage{epic,eepic}
\usepackage{color}



\theoremstyle{plain}

\newtheorem{Thm}{Theorem}[section]
\newtheorem{Lem}[Thm]{Lemma}

\theoremstyle{remark}
\newtheorem{Def}[Thm] {Definition}

\long\def\begcom#1\endcom{}

\newcommand{\Vol}{\operatorname{vol}}

\newcommand{\orb}{\operatorname{orb}}
\newcommand{\length}{\operatorname{\length}}

\newcommand{\vol}{\operatorname{vol}}

\newcommand{\dist}{\operatorname{dist}}
\newcommand{\cu}{\operatorname{\mathcal{U}}}
\newcommand{\cv}{\operatorname{\mathcal{V}}}
\newcommand{\cw}{\operatorname{\mathcal{W}}}
\newcommand{\cm}{\operatorname{\mathcal{M}}}
\newcommand{\ca}{\operatorname{\mathcal{A}}}

\def\supp{\operatorname{supp}}

\def\length{\operatorname{length}}

\begin{document}

\title[Approximation of Oseledets splittings  ]
      {Approximation of Oseledets splittings}

\author[Liang, Liao, Sun]
{Chao Liang$^{*}$, Gang Liao$^{\dag}$, Wenxiang Sun$^{\dag}$}

\thanks{$^{*}$ Applied Mathematical Department, The Central University of Finance and Economics,
Beijing 100081, China; Liang is supported by NSFC(\# 10901167)}\email{chaol@cufe.edu.cn}

 \thanks{$^{\dag}$ School of Mathematical Sciences,
Peking University, Beijing 100871,
China; Sun is supported by NSFC(\#
10231020) and Doctoral Education Foundation of China}\email{liaogang@math.pku.edu.cn}
\email{sunwx@math.pku.edu.cn}

\date{July, 2011}

\maketitle

\begin{center} \it Peking University,\,\,\,July, 2011 \end{center}

\begin{abstract}
We prove that the Oseledets splittings of an ergodic hyperbolic
measure of a $C^{1+ r}$ diffeomorphism can be approximated by that
of atomic measures on hyperbolic periodic orbits. This removes the
assumption on simple spectrum in \cite{Liang} and strengthens
Katok's closing lemma.
\end{abstract}

%\tableofcontents\newpage

\section{Introduction}

 Let $M$ be a
compact connected   $d$-dimensional Riemannian manifold without
boundary. For $1\leq i\leq d$, $x\in M$,  let $G(i,x)$ denote the
$i$-dimensional Grassmann space of $T_x M$. Given integers
$n_1,n_2\cdots,n_s$ with $\sum_{i=1}^sn_i=d$ and $1\leq n_s\leq d$,
we define  a bundle $\cu(n_1,\cdots,n_s)=\cup_{x\in
M}\cu(n_1,\cdots,n_s; x)$, where the fiber over $x$ is
$$\cu(n_1,\cdots,n_s; x)=\big{\{}(E_1, \cdots,E_{s})\in G(n_1,x)\times
\cdots \times G(n_s,x) \big{\}}.$$ Set
$\cv(n_1,\cdots,n_s)=\cup_{x\in M} \cv(n_1,\cdots,n_s; x)$, where
the fiber over $x$ is
$$\cv(n_1,\cdots,n_s; x)=\big{\{}(E_1, \cdots,E_{s})\in
\cu(n_1,\cdots,n_s)\mid \,E_1(x)\oplus\cdots\oplus E_{s}(x)=T_xM
\big{\}}.$$ It is immediate that by the definitions,
$\cv(n_1,\cdots,n_s)$ is not compact but
$$\cu(n_1,\cdots,n_s)=\overline{\cv(n_1,\cdots,n_s)}$$ is compact.

 Given integers $1\leq
l_1<l_2<\cdots<l_k<d$, we define  a filtration bundle
$\cw(l_1,\cdots,l_k)=\cup_{x\in M}U(l_1,\cdots,l_k; x)$, where the
fiber over $x$ is \begin{eqnarray*}\cw(l_1,\cdots,l_k;
x)=\big{\{}(F_1, \cdots,F_{k}; H_1, \cdots, H_k)\mid&& F_i\in
G(l_i,x),\,H_i\in G(d-l_i,x),\\[2mm]
&& F_i\subset F_{i+1}, \quad  H_{i+1}\subset H_{i}
\big{\}}.\end{eqnarray*} Clearly,  by the definitions,
$\cw(l_1,\cdots,l_k)$ is compact.

Throughout this paper, we assume the following relations
$$k=s-1,\quad l_i=\sum_{j=1}^in_j,\quad  1\leq i\leq s-1.$$

For the sake of statement, we appoint $\cu=\cu(n_1,\cdots,n_s)$,
$\cv=\cv(n_1,\cdots,n_s)$ and  $\cw=\cw(l_1,\cdots,l_k)$.  Given
$(E_1, \cdots,E_{s})\in \cv, $ we define $$\sigma ((E_1,
\cdots,E_{s}))=(E_1, E_1\oplus E_2, \cdots,
\oplus_{i=1}^{s-1}E_i;\quad \oplus_{i=2}^{s-1}E_i, \cdots,
E_{s-1}\oplus E_s, E_s)\in \cw.$$ Then $\sigma: \cv\to \cw$ is a
onto diffeomorphism. Denote $\cw_1=\sigma(\cv)$ and
$\sigma^{-1}=\sigma^{-1}\mid_{\cw_1}$.

 For a $p$-frame $\xi=(u_1,\cdots,u_p)$ which spans
a $p$-dimensional space $E$, we define $\Vol(\xi)=\Vol(u_1,
u_2,\cdots, u_p)$ the volume of the parallelepiped generated by the
vectors $u_1, u_2, \cdots, u_p$. More precisely, we choose an
orthonormal $p-$frame $\zeta=(w_1, w_2, \cdots, w_p),$ $w_i\in T_xM,
\,\,\,\,i=1, \cdots, p, $ which generates the same linear subspace
of $T_xM$ as $\xi$ does and take a unique $p\times p $ matrix $A$
with $\xi=\zeta A.$ Then we define the volume of $\xi$ by
$$\Vol(\xi):=|\det\,A|.$$
We remark that the volume  $\Vol(\xi)$ does not depend on the choice
of $\zeta$, since the determinate of a transition matrix between two
orthonormal frames is $\pm 1$. Hence, $\Vol(\xi)$ is well defined.
Given any linear onto transformation $\Phi$ from one linear space
$E$ to a linear space $F$, we define
$$\det(\Phi)=\frac{\vol(\Phi(\xi))}{\vol(\xi)},$$ where $\xi$ is a
frame spanning $E$. The definition does not depend on the choice of
$\xi$.
\bigskip

 Let $f$ be a $C^{1}$ diffeomorphism of  $M$.  For
 $1\leq i\leq k$,  define projections
$$\pi_i(\gamma)=F_i,\quad \widehat{\pi}_i(\gamma)=H_i,$$
and functions on $\cw(l_1,\cdots,l_k)$
$$\phi_i(\alpha)=\det(Df_x\mid_{ E_i}),\,\,\,\psi_i(\gamma)=\det(Df_x\mid_{ F_i}),$$
where $\gamma=(F_1,\cdots,F_k; H_1,\cdots,H_k)\in
\cw(l_1,\cdots,l_k; x).$ Then $\phi_i$, $\psi_i$ $(1\leq i\leq k)$
are bounded and continuous on $\cw(l_1,\cdots,l_k)$.

In what follows, suppose $f$ preserves an ergodic measure $\omega$.
Then there exist

\begin{enumerate}
\item[(a)]  real numbers  $\lambda_{1}< \cdot\cdot\cdot <
\lambda_{s}(s \leq d)$;

 \item[(b)]  positive integers $n_{1},
\cdot\cdot\cdot, n_{s}$, satisfying $n_{1}+ \cdot\cdot\cdot
+n_{s}=d$;

\item[(c)]  a Borel set $O(\omega)$, called  Oseledets
basin of $\omega$, satisfying $f(O(\omega))=O(\omega)$ and
$\omega(O(\omega))=1$;

\item[(d)]  a measurable splitting, called Oseledets splitting,
$T_{x}M=E_1(x)\oplus\cdot\cdot\cdot\oplus E_s(x)$ with
dim$E_i(x)=n_{i}$ and $Df(E_i(x))=E_i(fx)$,
\end{enumerate}
such that
$$\lim_{n\rightarrow\pm \infty}\frac{\log\|Df^{n}v\|}{n}=\lambda_{i},$$
for $\forall x \in O(\omega)$, $v \in E_i(x),\,\, i=1, 2, \cdots,s$.

\bigskip

  (a) and (d) in the Oseledets theorem allow
us to arrange the Oseledets splitting according to the increasing
order of the Lyapunov exponents. To avoid excessive terminology, we
will arrange the Oseledets splitting at every point in the Oseledets
basin in this way throughout this paper without explanation. We call
the measure $\omega$ hyperbolic if none of  its Lyapunov exponents
is zero.

From now on, suppose $f : M \rightarrow M$ is $C^{1+r}$
diffeomorphism. We give a quick review concerning some notions and
results of Pesin theory. We point the reader to \cite{KM}\cite{Po}
for more details.
\subsection{Pesin set}
\begin{Def}\label{Def6} Given $\alpha,\beta\gg\epsilon
>0$, and for all $k\in \mathbb{Z}^+$,  the Pesin block
$\Lambda_k=\Lambda_k(\alpha,\beta;\,\epsilon)$ consists of all
points $x\in M$ for which there is a splitting $T_xM=E_x^s\oplus
E_x^u$ with the invariance property $(Df^m)E_x^s=E_{f^mx}^s$ and
$(Df^m)E_x^u=E_{f^mx}^u$, and satisfying:\\
 $(a)~
\|Df^n|E_{f^mx}^s\|\leq  e^{\epsilon k}e^{-(\beta-\epsilon)
n}e^{\epsilon|m|}, \forall m\in\mathbb{Z}, n\geq1;$\\
 $(b)~
\|Df^{-n}|E_{f^mx}^u\|\leq  e^{\epsilon k}e^{-(\alpha-\epsilon)
n}e^{\epsilon|m|}, \forall m\in\mathbb{Z}, n\geq1;$ and\\
 $(c)~
\tan(Angle(E_{f^mx}^s,E_{f^mx}^u))\geq e^{-\epsilon
k}e^{-\epsilon|m|}, \forall m\in\mathbb{Z}.$
\end{Def}

\begin{Def}\label{Def7} $\Lambda(\alpha,\beta;\epsilon)=\overset{+\infty}{\underset{k=1}{\cup}}
\Lambda_k(\alpha,\beta;\epsilon)$ is a Pesin set.
\end{Def}
We say an invariant measure $\omega$ is related to the Pesin set $
\Lambda=\Lambda(\alpha,\beta;\epsilon)$,  if
$\omega(\Lambda(\alpha,\beta;\epsilon))=1$.

The following theorem is the main theorem in \cite{WS}.

\bigskip

\begin{Thm}{\rm\cite{WS}}\label{thm exponents approximation}
Let $M$ be a compact $d-$dimensional Riemannian manifold. Let $f:M
\rightarrow M$ be a $C^{1+r}$  diffeomorphism, and let $ \omega$ be
an ergodic hyperbolic measure with Lyapunov exponents
$\lambda_{1}\leq \ldots\leq \lambda_{r}<0<\lambda_{r+1}\leq
\ldots\leq \lambda_{d}.$ Then the Lyapunov exponents of $ \omega $
can be approximated by Lyapunov exponents of hyperbolic periodic
orbits. More precisely, for any $\varepsilon >0$, there exists a
hyperbolic periodic point $z$ with Lyapunov exponents
$\lambda_{1}^{z}\leq\ldots\leq\lambda_{d}^{z}$ such that
$\mid\lambda_{i}-\lambda_{i}^{z}\mid<\varepsilon,$ $ i=1, ..., d.$
\end{Thm}

\bigskip

\begin{Def}\label{DefMeanAng}
Let $E$ and $F$ be two  $Df-$invariant sub-bundles of $TM$. The
angle between $E(x)$ and $F(x)$, $x\in M$, is defined as follows:
$$\sin\angle(E(x),\,F(x)):=\inf\limits_{0\neq u\in E(x),0\neq v\in F(x)}\sin\angle(u,\,v)=
\inf\limits_{0\neq u\in E(x),0\neq v\in F(x)}\frac{\|u\wedge
v\|}{\|u\|\|v\|},$$
 where $\wedge$ denotes the wedge product. We call
$$m\angle(E(x), F(x)):=\lim_{n\to+\infty}\frac{1}{n}\sum^{n-1}_{i=0}\angle(Df^iE(x),\,Df^iF(x))$$
the mean angle between $E$ and $F$ at $x$.
\end{Def}

We view $T_xM$ as $\mathbb{R}^d$. Given $v\in \mathbb{R}^d$ and a
linear subspace $E\subset\mathbb{R}^d$, define
\begin{eqnarray*}\Gamma(v,E)=\begin{cases}\, \mbox{ the projection
 of} \,v/\|v\|\,\mbox{ onto} \,E,\,&v\neq0,\\
 \,0,&v=0.\end{cases}\end{eqnarray*}
 Furthermore, $\Gamma(F,E)$ is viewed as an operator taking values in $F\subset \mathbb{R}^d$. We adopt the classic metric $d_G$ on the Grassmann bundle
 $\cup_{0\leq i\leq d}G(i)$:
\begin{equation*}d_G(F,\,F')=\|\Gamma(\mathbb{R}^l,F)-\Gamma(\mathbb{R}^l,F')\|:=
\sup_{v\in \mathbb{R}^l}\|\Gamma(v,F)-\Gamma(v,F')\|,\end{equation*}
 where $F, F'\in \cup_{0\leq i\leq d}G(i)$. The definitions
above can be naturally extended to the translations of linear
subspaces. When $F, F'$ are disjoint, $d_G$ is equivalent to the
angle $\angle$.
\begin{Def}\label{DefMeanAng}
Let $E$ and $F$ be two $Df-$invariant sub-bundles of $TM$.  We call
$$md_G(E(x), F(x)):=\lim_{n\to+\infty}\frac{1}{n}\sum^{n-1}_{i=0}d_G(Df^iE(x),\,Df^iF(x))$$
the mean distance between $E$ and $F$ at $x$.
\end{Def}

Suppose that $f$ preserves an ergodic measure $\omega$ with its
Oseledec splitting
\begin{eqnarray}\label{equation:1.2}T_{x}M=E_{1}(x)\oplus\cdot\cdot\cdot\oplus
E_{s}(x),\quad s\leq d=\dim M, \quad x\in O(\omega). \end{eqnarray}
By the Birkhoff Ergodic Theorem, there is an $\omega-$full measure
subset in $O(\omega)$ such that every point $x$ in this subset
satisfies
$$md_G(E_i(x), E_j(x))=\int d_G(E_i(y), E_j(y))d\omega(y).$$
%Without any
%confusion, we can assume that the above equation holds for every
%point in $ O(\tilde\omega)$.
We call $\int d_G(E_i(y), E_j(y))d\omega(y)$ the mean distance
between $E_i:=\cup_{x\in O(\omega)} E_i(x)$ and $E_j:=\cup_{x\in
O(\omega)} E_j(x)$ and write it as $m{d_G}(\omega)(E_i, E_j),$
$\forall\, 1\leq i\ne j \le d.$

\bigskip

%so that $\lambda_i(\tilde\omega)=\lim_{n\to
%\infty}\frac 1n\log \|Df_x^n(v)\|,\,\,\,\,v\in E_i(x),\,\,\,\,i=1,
%\cdots, d.$

Given  $\gamma=(E_1,\cdots, E_t)\in T_xM\times \cdots \times T_xM,$
denote by $A(\gamma)$ the matrix
$$(\det(E_i)\det(E_j)\cos\angle(E_i,E_j))_{t\times t}.
$$ Let $\sigma(\gamma)$ denote the set of all eigenvalues of
$A(\gamma)$ and let $\tau(\gamma)$ be the smallest eigenvalue. Note
that $A(\gamma)$ is a real positive-definite symmetric matrix,
therefore, $\sigma(\gamma)\subset (0, +\infty).$

 For $x\in
O(\omega)$, we define the independence number of $x$ by the
independence number of bundles at $x$ whose elements are all on
different invariant bundles. More precisely, take
$\gamma(x)=(E_1(x),\,E_2(x),\,\cdots,\,E_s(x))$. Then we define
$$\tau(x):=\tau(\gamma), $$ the smallest eigenvalue of
$A(\gamma).$ Clearly, $\tau (x)$ is well defined for $x\in
O(\omega).$ Moreover, by the Birkhoff Ergodic Theorem, the equation
$$\lim_{n\rightarrow+\infty}\frac 1 n \sum_{i=0}^{n-1}\tau(f^ix)
=\int\tau(y)d\omega(y)$$ holds on an $\omega-$full measure subset of
$O(\omega)$. Therefore, we can define the independence number of
$\omega$ by
$$\tilde\tau(\omega):=\int\tau(y)d\omega(y).$$
%and call it independence number of $\tilde\omega$.
\bigskip

Assume there is another ergodic hyperbolic measure $\omega'$ with a
$Df$-invariant splitting
 \begin{equation}\label{equation:1.3}
T_yM=E_1(y)\bigoplus E_2(y)\bigoplus\cdot\cdot\cdot\bigoplus
E_{s}(y),\quad \dim(E_r)=n_r, \quad 1\leq r\leq s, \quad y\in
O(\omega').
\end{equation}
 Under these assumptions  we further describe the
approximation of Oseledec splittings. Remember both (\ref
{equation:1.2}) and (\ref{equation:1.3}) are arranged according to
the increasing order of Lyapunov exponents.

\begin{Def}\label{Defsplittingclose}
Let $\eta>0.$  The  Oseledec splitting (\ref {equation:1.2}) of
$\omega$ is $\eta$ approximated by (\ref {equation:1.3}) of
$\omega'$,  if there exists a measurable subset $\Gamma$ satisfying:
\begin{enumerate}

\item[(a).]\,
 $
\omega(\Gamma)>1-\eta$;

\item[(b).]\, for any $x\in \Gamma$, there exist a point $z=z(x)\in \supp(\omega')\cap O(\omega')$ and
$\beta(z)=(E_1(z), \cdots, E_s(z))\in \cv(n_1,\cdots,n_s;z)$ such
that
$$\dist(\gamma, \beta)<\eta,$$
where  the Oseledets bundle $\gamma(x)=(E_1(x), \cdots, E_{s}(x))\in
\cv(n_1,\cdots,n_s;x)$.
\end{enumerate}
\end{Def}

 Now we state our main result of this note.

\bigskip
\begin{Thm}\label{Thmmeanangleclose}
Let $f: M\to M$ be a $C^{1+r}$ diffeomorphism preserving an ergodic
hyperbolic measure $\omega$ with its Oseledets splitting
$$T_{\Lambda}M=E_1(\Lambda)\bigoplus
E_2(\Lambda)\bigoplus\cdot\cdot\cdot\bigoplus E_{s}(\Lambda),\quad
\dim(E_i)=n_i, \quad 1\leq i\leq s$$ arranged according to the
increasing order of Lyapunov exponents of $\omega$, where
$\Lambda=\bigcup_{k\geq1}\Lambda_{k}$ is the Pesin set associated
with $\omega$.
  Given $\varepsilon>0,$ there is a
hyperbolic periodic orbit $\orb(z,\, f)$ together with an invariant
splitting
$$T_{z}M=E_1(z)\bigoplus
E_2(z)\bigoplus\cdot\cdot\cdot\bigoplus E_{s}(z),\quad
\dim(E_i)=n_i, \quad 1\leq i\leq s$$ at $z$ arranged according to
the increasing order of Lyapunov exponents of the orbit $\orb(z)$
such that the atomic measure $\omega_z$ supported on $\orb(z,f)$
satisfies the following properties:
\begin{enumerate}
\item[(i).] Mean distance of $\omega$ and of $\omega_z$ are
$\varepsilon-$close, that is,
$$
|md_G(\omega)(E_i(\Lambda), E_j(\Lambda))
-md_G(\omega_z)(E_i(\orb(z)), E_j(\orb(z))) |<\varepsilon,\quad
\forall\,1\leq i\neq j\leq s;$$

\item[(ii).] Independence numbers of $\omega$ and of
$\omega_z$ are $\varepsilon-$close, that is,
$$|\tilde
\tau(\omega)-\tilde \tau(\omega_z)|<\varepsilon;$$

\item[(iii).] The Oseledec splitting of $\omega$ is
$\varepsilon-$approximated by that of $\omega_z$.
\end{enumerate}
\end{Thm}
\smallskip

\section{Proofs}
\bigskip

In the beginning, we will find an isolated ergodic measure on
$\cw(n_1,\cdots,n_s)$ covering $\omega$, which is needful in the the
following proofs.

\begin{Lem}\label{good ergodic measure covering omega}
There exists   one and only one invariant measure $m\in
\cm_{inv}(\cw,D^{\#}f)$ such that $\pi_{*}(m)=\omega$ and
$$\int_{\cw}\phi_idm=\sum_{j=1}^i n_j\lambda_j,\quad \int_{\cw}\psi_idm=\sum_{j=i+1}^s n_j\lambda_j,\quad 1\leq i\leq s-1.$$
In addition, $m$ is ergodic and $m(\cw_1)=1$.
\end{Lem}

\begin{proof}

By Oseledets theorem, take and fix a point $x\in Q_m(M,\, f)$,
$\gamma_0=(E_1(x),\cdots,E_s(x))\in \cv$ so that
$$\lim_{n\to\pm\infty} \frac 1n\Sigma_{i=0}^{n-1}\phi_t((Df^i\gamma_0)
=n_t\lambda_t,\,\,\,1\leq t\leq s.$$ Let
$\gamma_1=\sigma(\gamma_0)$. Define a sequence of measures $\mu_n$
on $\cw$ by
$$\int\,\phi d\mu_n:=\frac 1n \Sigma_{i=0}^{n-1}\phi(D^{\#}f^i\gamma_1),\,\,\,\,
\,\,\forall \phi\in C^0(\cw, \mathbb{R}).$$ By taking a subsequence
when necessary we can assume that $\mu_n\to \nu_0.$ It is standard
to verify that $\nu_0$ is a $D^{\#}f$-invariant measure and $\nu_0$
covers $m,$ i.e., $\pi_*(\nu_0)=\omega.$ Set
$$Q(\cw, D^{\#}f):=
\cup_{\nu \in \mathcal{M}_{erg}(\cw, D^{{\#}}f)} Q_\nu(\cw,
D^{\#}f).$$ Then $ Q(\cw, D^{\#}f)$ is a $D^{\#}f-$invariant total
measure subset in $\cw.$ We have
\begin{align*}
&\omega( Q_{\omega}(M, f) \cap \pi Q(\cw, D^{\#}f))\\
\ge & \nu_0(\pi^{-1}Q_{\omega}(M, f)\cap Q(\cw, D^{\#}f))\\
=&1.
\end{align*}
Then the set
\begin{align*} \mathcal{A}:=\{\mu\in \mathcal{M}_{erg}(\cw,
D^{\#}f)\,|\,\,&\exists \,\, \gamma\in
Q(\cw, D^{\#}f), \pi(\gamma)\in Q_{\omega}(M, f), s.\,t. \\
&\lim_{n\to+\infty}\frac 1n\Sigma_{i=0}^{n-1} \phi(D^{\#}f^i\gamma)
=\lim_{n\to-\infty}\frac 1n\Sigma_{i=0}^{n-1}
\phi(D^{\#}f^{i}\gamma)\\
&=\int _{\cw} \phi\,d\mu\,\,\,\,\,\,\quad\quad\quad \forall \phi\in
C^0(\cw, \mathbb{R})\,\}
\end{align*}
is non-empty. It is clear that $\mu $ covers $\omega$,
$\pi_*(\mu)=\omega,$ for all $\mu\in \mathcal{A}.$ And we claim that
$\mathcal{A}$ coincides with the set of all the measures in $
\mathcal{M}_{erg}(\cw, D^{\#}f)$ that cover $\omega$.  In fact, if
$\mu\in\mathcal{ M}_{erg} (\cw, D^{\#}f)$ covers $\omega$,
$\pi_*\mu=\omega, $ from the fact that $\mu(Q_\mu(\cw, D^{\#}f))=1$,
we have
\begin{align*}
&\omega( Q_{\omega}(M, f) \cap \pi Q_\mu(\cw, D^{\#}f))\\
\ge & \mu(\pi^{-1}Q_{\omega}(M, f)\cap Q_\mu(\cw, D^{\#}f))\\
=&1.
\end{align*}
Thus there is $\beta\in Q_\mu(\cw, D^{\#}f)$ with $\pi(\beta)\in
Q_{\omega}(M,f),$ which means $\mu\in \mathcal{A}.$ Therefore,
$$\mathcal{A}=\{\mu\in \mathcal{M}_{erg}(\cw, D^{\#}f)\,\,|\,\,
\pi_*(\mu)=\omega\}.$$
\bigskip
Assume the ergodic decomposition of $\nu_0$ is of form
$$\nu_0=\int_{\mathcal{M}_{erg}(\cw,D^{\#}f)}d\tau_{\nu_0}(m).$$
Then $\tau_{\nu_0}(\mathcal{A})=1$. Since $\mu_n\to\nu_0$ and
$\phi_i$ $(1\leq i\leq s)$ are continuous,
$$\int_{\cw}\phi_id\nu_0=\lim_{n\to+\infty}\int_{\cw}\phi_id\mu_n=\sum_{t=1}^{i}n_t\lambda_t.$$
Using the ergodic decomposition of $\nu_0$, we obtain
$$\int_{\mathcal{M}_{erg}(\cw,D^{\#}f)}\int_{\cw}\phi_idmd\tau_{\nu_0}(m)=\int_{\cw}\phi_id\nu_0= \sum_{t=1}^{i}n_t\lambda_t,\,\,\,\,1\leq i\leq s-1.$$
 Observe that for any $m\in \ca$, for $m$-almost $\gamma$
$$\int_{\cw}\phi_idm=\lim_{n\to+\infty}\frac1n\sum_{j=0}^{n-1}\phi_i(D^{\#}f^j\gamma)\geq\sum_{t=1}^{i}n_t\lambda_t,$$
and the equality holds if and only if
$\pi_i(\gamma(x))=\oplus_{t=1}^{i}E_t(x)$.
 Hence,   for $\tau_{\nu_0}$-almost $m\in \mathcal{A}$, for
$m$-almost $\gamma(x)\in \cw$,
$$\int_{\cw}\phi_idm=\lim_{n\to+\infty}\frac{1}{n}\sum_{j=0}^{n-1}\phi_i(D^{\#}f^j\gamma)=\sum_{t=1}^{i}n_t\lambda_t,\,\,\,1\leq i\leq s-1.$$
In the same manner, for any $m\in \ca$, for $m$-almost $\gamma$
$$\int_{\cw}\psi_idm=\lim_{n\to+\infty}\frac1n\sum_{j=0}^{n-1}\psi_i(D^{\#}f^j\gamma)\leq\sum_{t=i+1}^{s}n_t\lambda_t,$$
and the equality holds if and only if
$\widehat{\pi}_i(\gamma(x))=\oplus_{t=i+1}^{s}E_t(x)$. It follows
that
 for $\tau_{\nu_0}$-almost $m\in \mathcal{A}$, for
$m$-almost $\gamma(x)\in \cw$,
$$\int_{\cw}\psi_idm=\lim_{n\to+\infty}\frac{1}{n}\sum_{j=0}^{n-1}\psi_i(D^{\#}f^j\gamma)=\sum_{t=i+1}^{s}n_t\lambda_t,\,\,\,1\leq i\leq s-1.$$

Additionally, if  $m' \in \mathcal{M}_{inv}(\cw, D^{\#}f)$
satisfying that  $\pi_{*}(m')=\omega$ and
$$\int_{\cw}\phi_idm=\sum_{j=1}^i n_j\lambda_j,\quad \int_{\cw}\psi_idm=\sum_{j=i+1}^s n_j\lambda_j,\quad 1\leq i\leq s-1,$$
then by above proof we can see that for $m'$-almost $\gamma(x)\in
\cw$, $\pi_i(\gamma(x))=\oplus_{t=1}^{i}E_t(x)$,
$\widehat{\pi}_i(\gamma(x))=\oplus_{t=i+1}^{s}E_t(x)$, $1\leq i\leq
s-1$. So, $m'=m$. That is, $m$ is unique and ergodic and
$m(\cw_1)=1$.
\end{proof}

\bigskip

{\bf Proof of Theorem \ref{Thmmeanangleclose}}

Take a decreasing sequence $\{\varepsilon_n\}_{n=1}^{\infty}$ which
approaches zero. We further assume that
$$\varepsilon_1<\frac12\min\big{\{}|\lambda_i-\lambda_j|\mid \,1\leq i\neq j\leq s\big{\}}.$$
 Applying Theorem \ref{thm
exponents approximation}  to the ergodic measure $\omega$, we can
find
 a hyperbolic periodic point $z_n$ with period $p_n$ which satisfies the
 following properties:

\begin{enumerate}

\item[(a)]the atomic measure
$\omega_{n}=\frac{1}{p_n}\sum^{p_n-1}_{i=1}\delta_{f^iz_n}$ is
$\varepsilon_n-$close to the measure $\omega$ in the weak*-topology.

\item[(b)]
Without loss of generality we assume that $\omega_n$ has an
invariant splittings $E_1^n\oplus E_2^n\oplus \cdots \oplus E_{s}^n$
on $T_{\orb(z_n,f)}M$ such that $\dim(E_i^n)=n_i$ $(1\leq i\leq s)$
and for any $0\neq v\in E_i^n(z_n)$, the Lyapunov exponents
\begin{eqnarray}\label{small gap of Lyapunov exponnets}\lambda_i-\varepsilon_n<\liminf_{k\to
\infty}\frac1k\log\|Df^k_{z_n}v\|\leq \limsup_{k\to
\infty}\frac1k\log\|Df^k_{z_n}v\|
<\lambda_i+\varepsilon_n.\end{eqnarray}

\end{enumerate}

\bigskip
Let $\xi_n(x)=(E_1^n(x), E_2^n(x), \cdots  E_{s}^n(x))$ for $x\in
\orb(z)$. Thus, there is an ergodic invariant measure $m_n\in
\mathcal{M}_{erg}(\cw_1,D^{\#}f)$ such that $\pi_{*}(m_n)=\omega_n$
and
$$m_n(\sigma(\xi_n(f^tz_n)))=\frac{1}{p_n},\,\,\,\,\, 0\leq t\leq p_n-1.$$ At
most taking  a subsequence, we suppose that $m_n$ converges to an
invariant $\mu\in \cm_{inv}(\cw, D^{\#}f)$. Since
$\pi_{*}(m_n)=\omega_n$ and $\lim_{n\to +\infty}\omega_n=\omega$, so
$\pi_{*}(\mu)=\omega$.

By (\ref{small gap of Lyapunov exponnets}), it is easy to verify
that
\begin{eqnarray*}\sum_{j=1}^in_j(\lambda_j-\varepsilon_n)&<&\int_{\cw}\phi_idm_n<
\sum_{j=1}^in_j(\lambda_j+\varepsilon_n),\quad 1\leq i\leq
s-1,\\[2mm]
\sum_{j=i+1}^sn_j(\lambda_j-\varepsilon_n)&<&\int_{\cw}\psi_idm_n<
\sum_{j=i+1}^sn_j(\lambda_j+\varepsilon_n),\quad 1\leq i\leq
s-1.\end{eqnarray*} Letting $n\to \infty$, we deduce
\begin{eqnarray*}\int_{\cw}\phi_id\mu&=&\sum_{j=1}^in_j\lambda_j,\quad 1\leq i\leq
s-1,\\[2mm]\int_{\cw}\psi_id\mu&=&\sum_{j=i+1}^s n_j\lambda_j,\quad 1\leq
i\leq s-1.\end{eqnarray*} By Lemma \ref{good ergodic measure
covering omega}, it follows that $\mu=m$.

Denote $\rho_0=\sigma^{-1}_{*}(m)$, $\rho_n=\sigma^{-1}_{*}(m_n)$.
Then $\rho_0, \rho_n\in \cm_{erg}(\cv, D^{\#}f)$ and $\rho_n\to
\rho_0$ as $n\to +\infty$.

\noindent{\bf Proof of (i)}\,\,We need verify that
\begin{eqnarray}\label{mean angle}\,\,\,\,\,\lim_{n\to\infty}\frac{1}{p_{n}}\sum_{k=0}^{p_{n}-1}\angle(E_i^n(f^k(z_{n})),\,E_j^n(f^k(z_{n})))=
m\angle_{\omega}(E_i,\,E_j) ,\,\,1\leq i\neq j\leq s,\end{eqnarray}
where $E_1^n(\orb(z))\oplus\cdots \oplus E_s^n(\orb(z))$ is the
invariant splitting given as above (b). We define $\phi_{ij}: \cu\to
\mathbb{R}$, $(E_1,\cdots,E_s)\to \angle (E_i,E_j)$.  Then
$\phi_{ij}$ is bounded and continuous. Noting that $\rho_n\to
\rho_0$, we get (\ref{mean angle}).

\smallskip

\noindent{\bf Proof of (ii)}\,\,
 We recall the measures $\omega_n$ and
$\omega$  and $m_n$ and $m$ and their relations:
$$\omega_n\to\omega,\,\,\,\,\rho_n\to
\rho_0,\,\,\,\,\pi_*(\rho_n)=\omega_n,\,\,\,\,\pi_*(\rho_0)=
\omega.$$ Since the function $\tau:\,\cu\to\mathbb{R}$ is
continuous, so $$\int \tau\,d\rho_n\to \int \tau\,d\rho_0.$$
 Moreover,
the ergodic measure $m\in \mathcal{M}_{erg}(\cu,D^{\#}f)$ is unique
in the sense of Lemma \ref{good ergodic measure covering omega},
which implies
$$\int \tau\,d\rho_0=\int \tau\,d\omega.$$ Note that $\omega_n$ is an
ergodic measure whose spectrum are increasingly  arranged as
$E_1^n\oplus\cdots\oplus E_s^n$. This implies that
$$\int \tau\,d\rho_n=\int \tau\,d\omega_n.$$  Hence we have that $\int
\tau\,d\omega_n\to \int \tau\,d\omega$ or $\tilde\tau(\omega_n)\to
\tilde\tau( \omega).$ We thus obtain (ii).
\smallskip

\bigskip

\noindent{\bf Proof of (iii)}\,\, Given $\varepsilon>0$, take $l$
large, so that
$$ \omega(\Lambda_{l}( \omega ))>1-\varepsilon, \eqno(4.12)$$ where
$\Lambda_l(\omega)$ denotes the $l-th$ Pesin block associated with
$\omega $. Since the splitting
$$x\to
E_1(x)\oplus\cdots\oplus E_s(x)$$ depends continuously on
$x\in\Lambda_{l}( \omega )$, we can choose a uniform constant
$\eta>0$ satisfying
$$\eta<\min_{{ i\neq j}\atop{x\in\Lambda_l(\tilde \omega
)}}\{\frac{1}{10}\angle(E_i(x),\,E_j(x)),\,\varepsilon\}.\eqno(4.13)$$
For each   $x\in\Lambda_{l}(\omega )\cap \,\supp(\omega)$, we take
and fix   $\alpha_0(x)=(E_1(x),\cdots,E_s(x))$.  Denote by
$B(\alpha_0(x),\,\eta)$ the $\eta-$neighborhood of $\alpha_0(x)$
under  the   metric on the Grassman bundle. Then
$m(B(\alpha_0(x),\,\eta))>0.$ Recalling that $\rho_n\rightarrow
\rho_0$, we have
$$\liminf_{n\rightarrow+\infty}\rho_n(B(\alpha_0(x),\,\eta))\geq \rho_0(B(\alpha_0(x),\,\eta))>0.$$
Therefore, we can take  an integer $N(x)=N(\alpha_0(x))>0$ such that
$$\rho_{n}(B(\alpha_0(x),\,\eta))>0,\,\,\,\,\forall\, n\geq N(x).$$
This implies  the existence of  an element $\beta_n(x)$ in $
Q_{\rho_{n}}(\cu,D^{\#}f) \cap\, \supp(\rho_{n})\cap
B(\alpha_0(x),\,\eta)$ for each $ n\geq N(x).$ Observing that
$\rho_{n}$ covers the atomic measure $\omega_{n}$, we can deduce
that $\beta_{n}(x)$ must be an element based on a periodic point
$z(x,n)$
 on $\orb(z_{n})$, where $z_{n}$ is the periodic point
chosen in $(a)(b)$. By the uniqueness of $\rho_n$, we know that
$\beta_n(x)=(E_1^n(z(x,n)),\cdots,E_s^n(z(x,n))$.  Thus the
Oseledets splitting of $ \omega$ at $x$ is $\eta$ approximated  by
the invariant splitting $E_1^n(z(x,n))\oplus \cdots\oplus
E_s^n(z(x,n)$ of $\omega_n$ at a point $z=z(x,\,n)$ on $\orb(z_n),
\,\, n\geq N(x) .$ Note that the number $N(x)$ may vary
 with $x$, we need to find a number $N_1$, independent of
the choice of $x\in \Lambda_{l}(\omega )$, such that $\rho_{N_1}$
meets $(iii)$. This can be done by the compactness of $\Lambda_{l}(
\omega )$ and  continuity of the Oseledets splitting on it.

Hence we complete the proof of Theorem \ref{Thmmeanangleclose}.
\hfill$\Box$

\bigskip

\end{document}